\documentclass[11pt]{amsart}
\usepackage{amsthm}
\usepackage{amssymb}
\usepackage[left=3cm,right=2cm,top=2cm,bottom=2cm]{geometry}
\newtheorem*{thm*}{Theorem}
\newtheorem{lemma}{Lemma}[section]
\newtheorem{thm}[lemma]{Theorem}
\newtheorem{cor}[lemma]{Corollary}
\newtheorem{cond}[lemma]{Condition}
\theoremstyle{definition}

\DeclareMathOperator\E{E}
\DeclareMathOperator\G{G}
\DeclareMathOperator\GL{GL}
\DeclareMathOperator\Sp{Sp}
\newcommand\m{\mathfrak m}
\newcommand\N{\mathbb N}
\newcommand\Z{\mathbb Z}
\newcommand\blank{\mathord{\hbox to 1.5ex{\hrulefill}}\,}

\let\eps\varepsilon

\newcommand\supb[1]{\mathchoice
    {^{\fboxsep1pt\fbox{$\scriptstyle \mskip1mu #1\mskip-1mu$}}}%
    {^{\fboxsep1pt\fbox{$\scriptstyle \mskip1mu #1\mskip-1mu$}}}%
    {^{\fboxsep1pt\fbox{$\scriptscriptstyle \mskip.7mu #1\mskip-.7mu$}}}%
    {^{\fboxsep1pt\fbox{$\scriptscriptstyle \mskip.7mu #1\mskip-.7mu$}}}%
}
\title{Local-global principle for congruence subgroups of Chevalley groups}

\author{Himanee Apte}
\email{himanee200@gmail.com}

\author{Alexei Stepanov}
\email{stepanov239@gmail.com}
\address
{Department of Mathematics and Mechanics,
Sankt-Petersburg State University\newline
and\newline\hspace*\parindent
Abdus Salam School of Mathematical Sciences at
GC University, Lahore, Pakistan}

\thanks
{The article is written in the framework of Russian--Indian cooperation project RFBR~10-01-92651.
The second-named author gratefully acknowledges support of RFBR grant N. 11-01-00811-a and
State Financed research task 6.38.74.2011 at the St.Petersburg State University.}

\date{}

\begin{document}
\begin{abstract}
In this article we prove Suslin's local-global principle for principal congruence
subgroups of Chevalley groups. Let $\G(\Phi,\blank)$ be a Chevalley--Demazure group scheme
with a root system $\Phi\ne A_1$ and $E(\Phi,\blank)$ its elementary subgroup.
Let $R$ be a ring and $I$ an ideal of $R$. Assume additionally that $R$ has no residue
fields of $2$ elements if $\Phi=C_2$ or $G_2$.

\begin{thm*}
Let $g\in\G(\Phi,R[X],XR[X])$. Suppose that for every maximal ideal $\m$ of $R$
the image of $g$ under the localization homomorphism at $\m$
belongs to $\E(\Phi,R_\m[X],IR_\m[X])$. Then, $g\in\E(\Phi,R[X],IR[X])$.
\end{thm*}

The theorem is a common generalization of the result of E.Abe for the absolute case ($I=R$) and
H.Apte--P.Chattopadhyay--R.Rao for classical groups. It is worth mentioning that for the
absolute case the local-global principle was obtained by V.Petrov and A.Stavrova in more
general settings of isotropic reductive groups.
\end{abstract}
\maketitle

\section*{Introduction}

In this article we prove Suslin's local-global principle (LGP) for principal congruence
subgroups of Chevalley groups. The LGP in the context of lower
nonstable K-theory was introduced by Suslin~\cite{SuslinSerreConj} and Quillen~\cite{QuillenSerreConj}
for their solution of Serre's problem on projective modules over polynomial rings.
The localization techniques were developed further in works of
Bak~\cite{BakNonabelian}, Taddei~\cite{Taddei} and Vaserstein~\cite{VasersteinChevalley, VasersteinGLn}
among others. In the last decade a substantial progress in developing and applying
localization methods was made by Bak, Hazrat, Vavilov, Zhang Zuhong, and the second named
author~\cite{StepVavLength, BHVstrike, HSVZyoga, HVZrelachev}.
These papers utilize Bak's version of localization procedure where a key lemma is continuity
of conjugation in $s$-adic topology (Bak's key lemma).

Originally Suslin proved the LGP for the special linear group.
Later it was extended for orthogonal groups by Suslin--Kopejko~\cite{SusKopLGPSO}
and for symplectic groups by Grunewald--Mennicke--Vaserstein~\cite{GMVLGPSp}.
For all Chevalley groups the LGP was obtained by Abe~\cite{AbeLGP} with
an extra condition (condition (P) is used in the proof of Lemma~1.11, although it is
missing in the statement). For isotropic reductive groups the LGP was obtained by
Petrov and Stavrova~\cite{PetrovStavrovaIsotropic}.
The relative version of LGP in classical groups was worked out by the first-named author in cooperation
with Rao and Chattopadhyay~\cite{ApteChattRao}, where it was shown to be an important tool for the study
the group structure on orbit sets of unimodular rows. Rao and his group also found several new
applications of the LGP~\cite{ApteChattRao, ChattRao, BasuChattRao, BasuRao, RaoBasuJose, JoseRao}.

The LGP\footnote{In this article we deal with Suslin's version of the LGP, that serves for computing~$K_1$,
unlike Quillen's LGP concerned with $K_0$.}
was first applied for proving $K_1$-analogue of Serre's conjecture. The history of this statement
for classical groups is same as for the LGP itself. For all Chevalley groups a more general statement,
$\mathbb A^1$-homotopy invariance of nonstable $K_1$, was recently obtained by Wendt~\cite{WendtA1},
although  the proof is incomplete. In the more general settings of isotropic
reductive groups the homotopy invariance was independently proved by Stavrova~\cite{StavSerreConj}.

The key place of the proof of the LGP is the dilation principle (Theorem~\ref{RelDilation}).
It seems that the computations performed in section~\ref{CommCalc} of the current article are
absolutely necessary for the proof of the dilation principle. These computations were missing
in Abe's paper~\cite{AbeLGP}; he used extra conditions (P) and (P') instead.
Almost the same computations are needed for the proof of Bak's key lemma mentioned above.
The dilation principle as well as Bak's key lemma implies normality of the elementary subgroup
almost immediately. In fact, using ideas of preprint~\cite{StepUniloc} by the second-named author,
one can prove standard commutator formulas~\cite{VasersteinChevalley},
nilpotent structure of $K_1$~\cite{HazVav}, bounded
commutator length~\cite{StepVavLength}, and multiple commutator formulas~\cite{HazZhangMulti}
without further computations with group elements.

For the latter problem the relative
dilation principle obtained in the current article substantially simplifies
the proof. Actually, the dilation principle immediately implies Bak's key
lemma which was used in~\cite{HazZhangMulti} to get the result. On the other hand,
the commutator calculus of~\cite{HSVZyoga, HVZrelachev} provide formulas
which are not easy to get from the dilation principle.

Our proof does not depend on other results on the LGP in Chevalley groups.
The absolute case does not use normality of the elementary subgroup,
but for the relative case we apply standard commutator formulas in
section~\ref{CommForm}. We write all details in the proofs, although some
of them are obvious for specialists.

Those who want to focus on the absolute case
can set $I=A$ (or $I=R$ in the last section) and skip
sections~\ref{splitting} and \ref{CommForm}. Moreover, the case $\Phi=C_l$ is
proved by Grunewald--Mennicke--Vaserstein~\cite{GMVLGPSp}. Therefore, section~\ref{SpSection}
can also be omitted for the absolute case modulo this paper.

\section{Notation and preliminaries}\label{notation}

All rings in the article are assumed to be commutative with unity
and all ring homomorphisms preserve the unity.
Let $J$ be an ideal of a ring $A$. By $J^m$ we denote the product of
$J$ with itself $m$ times, i.\,e. $J^m$ is an additive subgroup generated
by all products $r_1\dots r_m$, where $r_1,\dots,r_m\in J$.

For a multiplicative subset $S$ of $A$ we denote by $S^{-1}A$
the localization of $A$ at $S$. The localization homomorphism $A\to S^{-1}A$
is denoted by $\lambda_S$.
In the current article we use only two kinds of localizations.
\begin{itemize}
\item
Principal localizations, where $S=\{s,s^2,\dots\}$ for some element $s\in A$.
Principal localization is denoted by $A_s$ and the corresponding localization
homomorphism by $\lambda_s$.
\item
Localizations at a maximal ideal. This means that $S=A\setminus\m$ for some maximal ideal
$\m$ of $A$. In this case localization is denoted by $A_\m$ and the localization
homomorphism by $\lambda_\m$.
\end{itemize}

Let $G$ be an arbitrary group. We follow standard group-theoretical notation:
\begin{itemize}
\item
if $a,b\in G$, then $[a,b]=a^{-1}b^{-1}ab$ denotes their commutator whereas
$a^b=b^{-1}ab$ stands for the conjugate to $a$ by $b$;
\item
if $F$ and $H$ are subsets of $G$, then $[F,H]$ is the mixed commutator subgroup, i.\,e. the
subgroup generated by commutators $[a,b]$ for all $a\in F$ and $b\in H$;
\item
the double commutators, i.\,e. $[F,H,K]:=\bigl[[F,H],K]$, where $F,H,K$ can be understood as
subsets or elements of $G$.
\end{itemize}

Several times we apply the 3 subgroup lemma which asserts that
\begin{equation}\label{3sbg}
[F,H,K]\le [K,F,H]\cdot [H,K,F]
\end{equation}
if $F,H,K$ are normal subgroups of $G$.

Let $\G(\Phi,\blank)$ denote a Chevalley-Demazure group scheme with a root
system~$\Phi$. We suppress the weight lattice from the notation. The point is that
we use only standard commutator formulas, which are known for all weight lattices
and elementary computations, i.\,e. computations in the Steinberg group.

Let $A$ be a ring. \textit{The following condition on $\Phi$ and $A$
is used throughout the article.}

\begin{cond}\label{cond}
$\Phi\ne A_1$. If $\Phi=C_2$ or $G_2$, then $A$ has no residue fields of\/ $2$ elements.
\end{cond}

Let $I$ be an ideal of~$A$.
The principal congruence subgroup $\G(\Phi,A,I)$ is the kernel of the reduction homomorphism
$\G(\Phi,A)\to\G(\Phi,A/I)$ modulo $I$.
By $\E(\Phi,I)$ we denote the subgroup of $\G(\Phi,A,I)$, generated by all root unipotent elements $x_\alpha(r)$
of level $I$ (which means that $r\in I$). If $J$ is another ideal of $I$, then
$\E(\Phi,J,IJ)$ denotes the normal closure of $\E(\Phi,IJ)$ in $\E(\Phi,J)$.
In particular, for $J=A$ we obtain the usual definition of the relative elementary subgroup
$\E(\Phi,A,I)$ of level~$I$.

The mixed commutator subgroup $[\E(\Phi,A,I),\E(\Phi,A,J)]$  plays an essential role in our
considerations. Since this subgroup appears very often, we introduce a special notation for it.

\begin{equation}\label{EIJ}
E_{I,J}=E_{I,J}(\Phi,A)=[\E(\Phi,A,I),\E(\Phi,A,J)]
\end{equation}

If $J=A$ and condition~\ref{cond} in satisfied, then the (absolute) standard commutator formulas
\begin{equation}\label{absolute}
[\G(\Phi,A,I),\E(\Phi,A)]=[\E(\Phi,A,I),\G(\Phi,A)]=[\E(\Phi,A,I),\E(\Phi,A)]=\E(\Phi,A,I)
\end{equation}
hold. They were obtained by Taddei~\cite{Taddei} and Vaserstein~\cite{VasersteinChevalley}.
The situation with relative standard commutator formula is more subtle. The inclusion
\begin{equation}\label{RelInclusion}
E_{I,J}\ge\E(\Phi,A,IJ)
\end{equation}
is known to fail for $\Phi=C_l$, if $2$ is not invertible in $A$.
(this follows from general statements about mixed commutators of relative subgroups in generalized
unitary groups, see~\cite[Lemmas~7,8]{HVZUniMulti} or~\cite[Lemma~23]{HVZUniRel}).
This exceptional%
\footnote{Exceptional groups are not exceptions in this settings, the only exception
is the classical group $\Sp_{2n}$.} case is considered in section~\ref{SpSection}.
The case of $\Sp_{2n}$ is already done in~\cite{ApteChattRao} provided that $2$ is invertible
in $A$, although the latter condition is omitted in ``Blanket assumptions'' of that paper.
Without this condition Lemma~3.6 of~\cite{ApteChattRao} is probably wrong.
Thus, section~\ref{SpSection} is to fill this gap.

Actually, result of section~\ref{SpSection} hold in general and they are sufficient for the proof of
the main theorems.
On the other hand, we want to state technical results of sections~\ref{splitting}--\ref{CommCalc}
in a sharp form as they are important themselves. Therefore, the following condition
is assumed in Lemmas~\ref{moving t}, \ref{inclusion}, and~\ref{decomposition}.

\begin{cond}\label{cond1}\strut
\begin{itemize}
\item
$\Phi\ne A_1$.
\item
If $\Phi=G_2$, then $A$ has no residue fields of\/ $2$ elements.
\item
If $\Phi=C_l$, then $2$ is invertible in $A$.
\end{itemize}
\end{cond}

\begin{lemma}[{\cite[Lemma~17]{HVZrelachev}}]\label{RelInc}
Under condition~\ref{cond1} inclusion~(\ref{RelInclusion}) is valid.
\end{lemma}

The reverse inclusion to~\ref{RelInclusion} is valid if $I+J=A$
(see~\cite[Theorem~5]{VavStepCommForm} for the case $\Phi=A_l$),
in general a counterexample was constructed in~\cite{MasonComm2}.
On the other hand, under condition~\ref{cond1} the formula
\begin{equation}\label{relative}
[\G(\Phi,A,I),\E(\Phi,A,J)]=[\E(\Phi,A,I),\E(\Phi,A,J)]
\end{equation}
follows from 3 subgroup lemma (formula~(\ref{3sbg})) and absolute standard commutator formulas~(\ref{absolute})
almost immediately (see the proof of~\cite[Theorem~2]{VavStepCommForm}).
We reproduce the proof in section~\ref{CommForm} for completeness.

To prove the relative dilation principle and the relative LGP, we shall use only absolute
standard commutator formulas~(\ref{absolute}).

We call $J$ a splitting ideal if $A=R\oplus J$ as additive groups, where $R$ is a subring
of $A$. Of course, in this case $R\cong A/J$. Equivalently, $J$ is a splitting ideal iff
it is a kernel of a retraction $A\to R\subseteq A$. In the current article we use
only one example of splitting ideals. Namely, if $A=R[t]$ is a polynomial ring,
then the ideal $tA$ is obviously splitting.

The following statement is called (absolute) splitting principle.
It follows easily from Lemma~\ref{GroupSplitting} (which itself is very simple).
It was formulated in~\cite[Proposition~1.6]{AbeLGP} and~\cite[Lemma~3.3]{ApteChattRao}.

\begin{lemma}\label{AbsoluteSplitting}
If $J$ is a splitting ideal of a ring $A$, then
$
\E(\Phi,A)\cap\G(\Phi,A,J)=\E(\Phi,A,J).
$
\end{lemma}

\section{Relative splitting principle}\label{splitting}

Let $J$ be a splitting ideal of $A$.
For relative groups absolute splitting principle~\ref{AbsoluteSplitting} implies that
$$\E(\Phi,A,I)\cap\G(\Phi,A,J)=\E(\Phi,A,I)\cap\E(\Phi,A,J).$$
For the relative dilation principle (Theorem~\ref{RelDilation}), a smaller group
on the right hand side of the above formula is required. Hence to aquire the desired statement there, we must work with an extended
ideal $I$ of $A$ induced from the retraction $A/J\to A$, i.\,e. $I=I'A$ for some ideal $I'$ in the image
of $A/J$.

The splitting principle based on the following simple group-theoretical lemma.
The proof is an easy group theoretical exercise and will be left to the reader.

\begin{lemma}\label{GroupSplitting}
Let $G=G'\ltimes G''$ be a semidirect product (so that $G''$ is a normal subgroup).
Let $H'\le G'$ and $H''\le G''$ be such that $H=H'H''$ is a subgroup (this amounts to say
that $H'$ normalizes $H''$). Then $H''=H\cap G''$.
\end{lemma}

To state the relative splitting principle in a form
independent of condition~\ref{cond1}
we introduce the following notation: $\bar E_{I,J}=E_{I,J}\E(\Phi,A,IJ)$.
By Lemma~\ref{RelInc}, under condition~\ref{cond1} we have $\bar E_{I,J}=E_{I,J}$.
Since relative elementary subgroup by definition is normal in $E(R)$,
the group $\bar E_{I,J}$ is normal in $E(\Phi,R)$ and
$\bar E_{I,J}\le E(\Phi,R,I)\cap E(\Phi,R,J)$.

\begin{lemma}[Relative Splitting Principle]\label{RelativeSplitting}
Suppose that $J$ is a splitting ideal of $A$ so that $A=R\oplus J$ for some subring $R$ of $A$.
Let $I'$ be an ideal of $R$. Put $I=AI'$. Then
$$\E(\Phi,A,I)\cap\G(\Phi,A,J)=\bar E_{I,J}.$$
\end{lemma}

\begin{proof}
Applying the functor $\G(\Phi,\blank)$ to $R\hookrightarrow A\twoheadrightarrow R$ we get a retraction
$\G(\Phi,R)\hookrightarrow \G(\Phi,A)\twoheadrightarrow \G(\Phi,R)$ and the kernel of the latter map is
$\G(\Phi,A,J)$. Therefore, $\G(\Phi,A)=\G(\Phi,R)\ltimes\G(\Phi,A,J)$.
Clearly, $\E(\Phi,R,I')\le\G(\Phi,R)$, $\bar E_{I,J}\le \G(\Phi,A,J)$, and
$E_{I,J}$ is normal in $\E(\Phi,A)$ (by definition of relative elementary subgroup).
By Lemma~\ref{GroupSplitting} it suffices to show that
$$\E(\Phi,A,I)=\bar E_{I,J}\E(\Phi,R,I').$$
Since $\E(\Phi,A,I)$ is normal in $\E(\Phi,A)$, the right hand side is contained
in the left hand side.

Note that the ideal~$I$ is additively generated by the elements~$rs$, where $r\in A$ and $s\in I'$.
Since $A=R\oplus J$ we can write $r=r'+r''$ for some $r'\in R$ and $r''\in J$.
Then $rs=r's+r''s\in I'+IJ$. Since $R\cap J=\{0\}$ we have $I=I'\oplus IJ$.

The group $\E(\Phi,A,I)$ is generated by the elements $x_\alpha(u+v)^b$, where $\alpha\in\Phi$,
$u\in I'$, $v\in IJ$, and $b\in\E(\Phi,A)$. Since $x_\alpha(v)^b\in\E(\Phi,A,IJ)\le\bar E_{I,J}$,
it remains to show that $x_\alpha(u)^b\in E_{I,J}\E(\Phi,R,I')$.
By the absolute splitting principle we can write $b=cd$, where $c\in\E(\Phi,R)$ and $d\in\E(\Phi,A,J)$.
Then $g=x_\alpha(u)^c\in\E(\Phi,R,I')$ and $x_\alpha(u)^b=g^d=[d,g^{-1}]\,g\in E_{I,J}\E(\Phi,R,I')$.
\end{proof}

\section{Commutator formulas}\label{CommForm}
The only commutator calculations with particular elements are used in sections~\ref{CommCalc}
and~\ref{SpSection}.
In the current section we perform some commutator calculus with subgroups in the spirit
of~\cite{VavStepCommForm}.

Let $H$ be an arbitrary algebraic group scheme over a ring $Z$ and $A$ a $Z$-algebra
(for a Chevalley group scheme one can take $Z=\Z$, the notation is to illustrate this).
Denote by $H(A,I)$ the kernel of the reduction homomorphism $H(A)\to H(A/I)$.
In the following lemma, we use a faithful representation of the group scheme $H$.
A faithful representation is an injective homomorphism with a special choice of generators of the
affine algebra of $H$. In fact the lemma can be proved by the clever choice of set of generators, but
the arguments using representation look more clear.
For $H=\GL_n$ the lemma was obtained in~\cite{VavStepCommForm0}.
Throughout this section $I$ and $J$ denote arbitrary ideals of a ring $A$.

\begin{lemma}\label{GIJ}
$[H(A,I),H(A,J)]\le H(A,IJ)$.
\end{lemma}

\begin{proof}
Let $H\to\GL_n$ be a faithful rational representation of~$H$.
Identifying elements of $H(A)$ with their images in $\GL_n(A)$ we have
$H(A,I)=H(A)\cap\GL_n(A,I)$. By~\cite[Lemma~3]{VavStepCommForm0}
the statement holds for $H=\GL_n$. Thus,
$$
[H(A,I),H(A,J)]\le H(A)\cap[\GL_n(A,I),\GL_n(A,J)]\le H(A)\cap\GL_n(A,IJ)=H(A,IJ).
$$
\end{proof}

Now we can prove equation~\ref{relative}. The proof is the same as for~\cite[Theorem~2]{VavStepCommForm}.

\begin{lemma}\label{RelCommForm}
Under condition~$\ref{cond}$ we have
$$E_{I,J}\le[\G(\Phi,A,I),\E(\Phi,A,J)]\le \bar E_{I,J}.$$
Under condition~$\ref{cond1}$ all inequalities turn into equalities.
\end{lemma}

\begin{proof}
The former inequality is trivial.
By condition~\ref{cond} we have absolute standard commutator formulas~(\ref{absolute}).
Therefore,
\begin{multline*}
[\G(\Phi,A,I),\E(\Phi,A,J)]=\bigl[\G(\Phi,A,I),[\E(\Phi,A,J),\E(\Phi,A)]\bigr]\le
\qquad\text{(by 3 subgroup lemma)}\\
[\G(\Phi,A,I),\E(\Phi,A,J),\E(\Phi,A)]\cdot[\E(\Phi,A),\G(\Phi,A,I),\E(\Phi,A,J)]\le
\qquad\text{(by Lemma \ref{GIJ})}\\
[\G(\Phi,A,IJ),\E(\Phi,A)]\cdot[\E(\Phi,A,I),\E(\Phi,A,J)]=\E(\Phi,A,IJ)E_{I,J}
\end{multline*}
\end{proof}

Let $I,J,K$ be ideals of a commutative ring $A$. The following useful fact
was observed for $GL_n$ in~\cite{HazZhangMulti}.

\begin{lemma}\label{EEE}
Under condition~$\ref{cond}$ we have
$$E_{IJ,K}\le[\bar E_{I,J},\E(\Phi,A,K)]\le\bar E_{IJ,K}.$$
Under condition~$\ref{cond1}$ all inequalities turn into equalities.
\end{lemma}

\begin{proof}
All the subgroups involved in the statement are normal, so we have
$[\bar E_{I,J},\E(\Phi,A,K)]=[E_{I,J},\E(\Phi,A,K)]\cdot E_{IJ,K}$.
Therefore, it remains to show that $[E_{I,J},\E(\Phi,A,K)]\le\bar E_{IJ,K}$,
but this follows immediately from lemmas~\ref{GIJ} and~\ref{RelCommForm}.
\end{proof}

From now on we assume that condition~\ref{cond1} holds until the middle of
section~\ref{CommCalc}.

The following statement is the first ingredient for the dilation principle.
In the proof of the dilation principle it is used with $J=At$, where
$A=R[t]$ is a polynomial ring. The lemma allows to
move an independent variable $t$ from the first subgroup of the mixed commutator
to the second one.

\begin{lemma}\label{moving t}
Let $I,J,K$ be ideals of $A$. Under condition~$\ref{cond1}$ we have
$E_{IJ^2,K}\le\E(\Phi,IJK)E_{IJ,KJ}$.
\end{lemma}

\begin{proof}
By Lemma~\ref{RelInc} for any pair $L,M$ of ideals
we have $\E(\Phi,A,LM)\le[\E(\Phi,A,L),\E(\Phi,A,M)]$.
Therefore,

{\parindent=0pt\parskip=3pt plus1pt minus 2pt
$\displaystyle E_{IJ^2,K}=[\E(\Phi,A,IJ^2),\E(\Phi,A,K)]\le [\E(\Phi,A,IJ),\E(\Phi,A,J),\E(\Phi,A,K)]\le$

\hbox to\textwidth{\hfill (by the 3 subgroups lemma)}

$\displaystyle [\E(\Phi,A,K),\E(\Phi,A,IJ),\E(\Phi,A,J)]\cdot[\E(\Phi,A,J),\E(\Phi,A,K),\E(\Phi,A,IJ)]\le$

\hbox to\textwidth{\hfill (by Lemma~\ref{EEE})}

$\displaystyle [\E(\Phi,A,KIJ),\E(\Phi,A,J)]\cdot[\E(\Phi,A,KJ),\E(\Phi,A,IJ)]\le
\E(\Phi,IJK)[\E(\Phi,A,IJ),\E(\Phi,A,KJ)].$
}
\end{proof}

\section{Commutator calculus}\label{CommCalc}

The following lemma is the second ingredient of
the dilation principle. For the absolute case it was formulated (without proof)
by J.Tits in~\cite{TitsCongruence}.

\begin{lemma}\label{inclusion}
Under condition~$\ref{cond1}$ we have
$\E(\Phi,A,IJ^2)\le \E(\Phi,J,IJ)$.
\end{lemma}

The proof of the above statement is based on the following three lemmas.
Let $\alpha\in\Phi$. Denote by $Y(\alpha,J,I)$ the subgroup of $\E(\Phi,A)$ generated by
elements $x_\beta(-r)x_\gamma(s)x_\beta(r)$ for all $\beta,\gamma\in\Phi\setminus\{\pm\alpha\}$,
$r\in J$ and $s\in I$. The next statement shows that we can decompose an elementary root
unipotent element of level $IJ^2$ to a product of elementary unipotents, avoiding
a given root and its inverse (the case $\Phi=C_l$ and $2$ is not invertible in $A$
is exceptional as above).

\begin{lemma}\label{decomposition}
Let $\alpha\in\Phi$. If condition~\ref{cond1} holds, then $x_\alpha(IJ^2)\le Y(\alpha,J,IJ)$.
\end{lemma}

\begin{proof}
Take $\beta,\gamma\in\Phi$ such that $\alpha=\beta+\gamma$. For any $p\in I$ and $q,r\in J$
write the Chevalley commutator formula
$$
[x_\beta(pq),x_\gamma(\eps r)]=x_\alpha(N_{\beta\gamma11}\eps pqr)
\prod x_{i\beta+j\gamma}(N_{\beta\gamma ij}\eps^j p^iq^ir^j),
$$
where $\eps$ is an invertible element of $A$ and the product is taken over all $i,j>0$
such that $i\beta+j\gamma\in\Phi$ and $i+j>2$.
Clearly, $\beta$, $\gamma$ and $i\beta+j\gamma$ with $i+j>2$ cannot be equal to $\pm\alpha$.
Therefore, the left hand side and each factor of the product belong to $Y(\alpha,J,IJ)$.
Hence, it suffices to find $\beta$ and $\gamma$ such that
$N_{\beta\gamma11}$ is invertible in $A$ and put $\eps=N_{\beta\gamma11}^{-1}$.

If $\alpha$ is a short root, then it is a sum of a short and long root.
If all roots have the same length, we call them long. In this case
$\alpha$, $\beta$ and $\gamma$ span the root system of type $C_2$ or $G_2$ and
$N_{\beta\gamma11}=\pm1$.

If $\alpha$, $\beta$ and $\gamma$ are long roots, then $N_{\beta\gamma11}=\pm1$.
Such a decomposition for a long root is available in all root systems except $\Phi=C_l$.
By condition~\ref{cond1}, if $\Phi=C_l$, then $2$ is invertible in $A$.
Any long root $\alpha\in C_l$ is a sum of two short roots $\beta$ and $\gamma$.
Now, the formula $[x_\beta(pq),x_\gamma(r/2)]=x_\alpha(pqr)$ completes the proof.
\end{proof}

The rest of the section does not depend on condition~\ref{cond1}.

\begin{lemma}[{Vaserstein~\cite[Theorem~2]{VasersteinChevalley}}]\label{generators}
Under condition~\ref{cond} the group $\E(\Phi,A,I)$ is generated by the elements
$x_{-\alpha}(-r)x_\alpha(s)x_{-\alpha}(r)$,
for all $\alpha\in\Phi$, $r\in A$ and $s\in I$.
\end{lemma}

The next lemma follows easily from the Chevalley commutator formula.

\begin{lemma}\label{simple conjugation}
For any $\beta\ne\alpha\in\Phi$, $r\in A$ and $s\in I$ we have
$$x_{-\alpha}(-r)Y(\alpha,J,IJ)x_{-\alpha}(r)\subseteq\E(\Phi,J,IJ)$$.
\end{lemma}
\begin{proof}
Any element in $x_{-\alpha}(-r)Y(\alpha,J,IJ)x_{-\alpha}(r)$ is $x_{-\alpha}(-r)x_\beta(-u)x_\gamma(s)x_\beta(u)x_{-\alpha}(r)$ with
$r\in A,u\in J,s\in IJ$ which also is an element in the subgroup of $\E(\Phi,IJ)$ conjugated by $\E(\Phi,J)$ and hence is in
$\E(\Phi,J,IJ)$.
\end{proof}

\begin{proof}[Proof of Lemma~$\ref{inclusion}$]
By Lemma~\ref{generators} it suffices to prove that
$x_{-\alpha}(-r)x_\alpha(s)x_{-\alpha}(r)\in \E(\Phi,J,IJ)$ for
all $\alpha\in\Phi$, $s\in IJ^2$, and $r\in A$. By Lemma~\ref{decomposition}
$x_\alpha(s)\in Y(\alpha,J,IJ)$.
Now, the result follows from Lemma~\ref{simple conjugation}.
\end{proof}

\section{Symplectic group}\label{SpSection}

In this section we make necessary changes in Lemmas~\ref{RelInc}, \ref{moving t}, \ref{decomposition},
and~\ref{inclusion} for the exceptional case $\Phi=C_l$.
For the same, we need to distinguish between $J^m$ and
$J\supb m$, which denotes the ideal generated by $r^m$ for all
$r\in J$. The following lemma is commonly known. We give a one row proof for
completeness.

\begin{lemma}\label{r2-r}
Suppose that a ring $A$ has no residue fields of $2$ elements. Then the
ideal generated by $r^2-r$ for all $r\in A$ coincides with the whole ring.
\end{lemma}

\begin{proof}
Suppose that there exists a maximal ideal $\m$ such that for all $r\in A$, we have $r^2-r\in\m$. Then $A/\m$ is a
field consisting of idempotents. Hence, $A/\m=\mathbb F_2$, a contradiction.
\end{proof}

The following statement substitutes Lemma~\ref{RelInc}. The idea of the proof is borrowed
from~\cite{HVZrelachev}.

\begin{lemma}\label{SpRelInc}
Under condition~\ref{cond} we have
$\E(\Phi,A,I\supb2J+2IJ+IJ\supb2)\le [\E(\Phi,A,I),\E(\Phi,A,J)]$.
\end{lemma}

\begin{proof}
In view of Lemma~\ref{RelInc} it suffices to consider the case $\Phi=C_l$.
Denote by $H$ the right hand side of the inclusion.
Since $H$ is normal in $\E(C_l,A)$, it remains to show that
$x_\alpha(r)$ belongs to the right hand side for all $\alpha\in C_l$ and
$r\in I\supb2J+2IJ+IJ\supb2$.

If $l\ge3$, then any short root lies in a subsystem of roots of type $A_2$.
For $A_2$ we have even stronger inclusion~(\ref{RelInclusion}).

Now, let $\alpha$ be a long root. Then there exists a long root $\beta$ and
a short root $\gamma$ such that $\alpha=\beta+2\gamma$. Note that $\beta+\gamma$ is
a short root. By the
formula $x_\alpha(2pq)=[x_{\beta+\gamma}(\pm p),x_\gamma(\pm q)]$ with $p\in I$ and $q\in J$
we get $x_\alpha(2IJ)\le H$. Then, apply the relation
$x_\alpha(pq^2)=[x_\beta(\pm p),x_\gamma(\pm q)]x_{\beta+\gamma}(\pm pq)$.
By the first paragraph of the proof $x_{\beta+\gamma}(\pm pq)\in H$.
Clearly, the first factor of the right hand side lies in $H$ provided that
$p\in I$, $q\in J$ or $q\in I$, $p\in J$. This shows that
$x_\alpha(I\supb2J)\le H$ and $x_\alpha(IJ\supb2)\le H$.

Now, let $l=2$, $\alpha$ and $\beta$ are long root, $\gamma$ is a short root, and
$\alpha=\beta+2\gamma$. Take $p\in I$, $q\in J$ and $r,s\in A$. Write Chevalley commutator formulas
\begin{align*}
[x_\beta(\pm sp),x_\gamma(\pm rq)]&=x_{\beta+\gamma}(\pm rpq)x_\alpha(sr^2pq^2)\in H\text{ and}\\
[x_\beta(\pm srp),x_\gamma(\pm q)]&=x_{\beta+\gamma}(\pm rpq)x_\alpha(srpq^2)\in H.
\end{align*}
The difference between right hand sides of these formulas is equal to
$x_\alpha\bigl(s(r^2-r)pq^2\bigr)$ and lies in $H$ for all $r\in R$.
By Lemma~\ref{r2-r} $x_\alpha(pq^2)\in H$. Similarly, $x_\alpha(qp^2)\in H$.
The proof of inclusion $x_\alpha(2IJ)\le H$ is the same as for $l\ge 3$.

Finally, displayed formulas above with $r=s=1$ shows that $x_{\beta+\gamma}(pq)\in H$
as the last factor is in~$H$.
\end{proof}

The next result is a substitute for Lemma~\ref{moving t}.
The proof is essentially the same.

\begin{lemma}\label{SpMoving t}
Let $I,J,K$ be ideals of $A$. Then
$\bar E_{IJJ\supb2,K}\le\E(\Phi,A,IJK)\cdot E_{IJ,KJ}$.
\end{lemma}

\begin{proof}
Clearly, $\E(\Phi,A,IJJ\supb2K)\le\E(\Phi,A,IJK)$.
By Lemma~\ref{SpRelInc} we have $\E(\Phi,A,IJJ\supb2)\le E_{IJ,J}$.
Now the proof almost coincides with the proof of Lemma~\ref{moving t}.
Namely, by the 3 subgroup lemma we have
$[E_{IJ,J},\E(\Phi,A,K)]\le[E_{IJ,K},\E(\Phi,A,J)]\cdot[E_{J,K},\E(\Phi,A,IJ)].$
Now, by Lemma~\ref{RelCommForm} the right hand side of the above formula is contained in
$\bar E_{IJK,J}\bar E_{JK,IJ}\le\E(\Phi,A,IJK)E_{JK,IJ}$.
\end{proof}

Now we take care of Lemma~\ref{inclusion}. First we prove a counterpart of
Lemma~\ref{decomposition}.

\begin{lemma}\label{SpDecomposition}
Let $\alpha\in\Phi$. Then $x_\alpha(IJJ\supb2)\le Y(\alpha,J,IJ)$.
\end{lemma}

\begin{proof}
In view of Lemma~\ref{decomposition} it suffices to consider the case $\Phi=C_l$.
The proof is the same as for the previous lemma replacing $I$ by $IJ$. Details are
left to the reader.
\end{proof}

\begin{lemma}\label{SpInclusion}
$\E(\Phi,A,IJJ\supb2)\le \E(\Phi,J,IJ)$.
\end{lemma}

\begin{proof}
By Lemma~\ref{generators} it suffices to prove that
$x_{-\alpha}(-r)x_\alpha(s)x_{-\alpha}(r)\in \E(\Phi,J,IJ)$ for
all $\alpha\in\Phi$, $s\in IJJ\supb2$ and $r\in A$. By Lemma~\ref{SpDecomposition}
$x_\alpha(s)\in Y(\alpha,J,IJ)$.
Now, the result follows from Lemma~\ref{simple conjugation}.
\end{proof}

\section{Relative dilation principle}\label{dilation}

In this section, we introduce a new indeterminate $t$. The next statement follows immediately from Lemmas~\ref{SpMoving t}
and~\ref{SpInclusion}.

\begin{cor}\label{ForDilation}
Let $K$ be an ideal in $A$.
Under condition~$\ref{cond}$ we have $\bar E_{K,At^9}\le\E(\Phi,At,Kt)$.
\end{cor}

\begin{proof}
Lemma~\ref{SpMoving t} with $I=A$ and $J=At^3$ shows that $\bar E_{K,At^9}$ is contained in
$\E(\Phi,A,Kt^3)E_{At^3,Kt^3}\le\E(\Phi,A,Kt^3)$. The latter group is contained in
$\E(\Phi,At,Kt)$ by Lemma~\ref{SpInclusion}.
\end{proof}

Let $R$ be a ring and $s\in R$.
Using the relative splitting principle and the previous corollary, it is easy to see that
if $g(t)\in\E(\Phi,R_s[t],IR_s[t])\cap\G(\Phi,R_s[t],tR_s[t])$, then
$g(s^nt)$ lies in the image of $\E(\Phi,R[t],IR[t])$ under the localization
homomorphism for some $n\in\N$. But localization homomorphism is not injective in general.
The following lemma is to get around this problem.
For the general linear group this trick was noticed in~\cite{SuslinSerreConj}.
In~\cite[Lemma~14]{PetrovStavrovaIsotropic} this is stated for
all affine group schemes. However, we believe that our detailed proof may be useful for
non-specialists. Recall that $\lambda_s$ denotes the principal localization homomorphism
$R\to R_s$, where $s\in R$.

\begin{lemma}\label{ZeroDivisors}
Let $H$ be an algebraic group scheme. Let $g(t),h(t)\in H(R[t],tR[t])$ and $s\in R$ be
such that $\lambda_s(g)=\lambda_s(h)$. Then there exists $m\in\N$ such that $g(s^mt)=h(s^mt)$.
\end{lemma}

\begin{proof}
As in Lemma~\ref{GIJ} we consider a faithful representation of the group scheme $H$ and
identify elements of $H(R)$ with their images in $\GL_n(R)$ for any ring $R$.
Recall that $a_{ij}$ denotes the entry of a matrix $a$ in position $(i,j)$.
Note that $\lambda_s(g)=\lambda_s(h)$ implies that there exists $m\in\N$ such that
$s^mg_{ij}=s^mh_{ij}$ for all $i,j=1,\dots,n$. Since
$g,h\in H(R[t],tR[t])\subseteq\GL_n(R[t],tR[t])$,
we have $g_{ij}-\delta_{ij}=t\tilde g_{ij}$ and $h_{ij}-\delta_{ij}=t\tilde h_{ij}$ for some
$\tilde g_{ij},\tilde h_{ij}\in R[t]$. Since $t$ is not a zero divisor,
$s^m\tilde g_{ij}=s^m\tilde h_{ij}$ for all $i,j=1,\dots,n$ which implies the result.
\end{proof}

\begin{thm}[Relative Dilation Principle]\label{RelDilation}
Let $g(t)\in\G(\Phi,R[t],tR[t])$ be such that $\lambda_s(g)\in\E(\Phi,R_s[t],IR_s[t])$.
Then there exists $l\in\N$ such that $g(s^lt)\in\E(\Phi,R,IR[t])$.
\end{thm}

\begin{proof}
Let $t'$ be another independent variable and $R'=R_s[t,t']$.
Put $f(t,t')=\lambda_s\bigl(g(tt')\bigr)$.
Then $f(t,t')\in\E(\Phi,R',IR')\cap\G(\Phi,R',tR')$.
By Lemma~\ref{RelativeSplitting} $f(t,t')\in\bar E_{IR',t'R'}(\Phi,R')$ and
by Corollary~\ref{ForDilation} $f(t,{t'}^9)\in\E(\Phi,t'R',t'IR')$.
By the definition of the latter group,
$$
f(t,{t'}^9)=\prod_{j}x_{\alpha_j}\left( \frac{t'v_j}{s^{k_j}}\right)^{a_j}, \text{ where }
a_j=\prod_i x_{\alpha_{ji}}\left( \frac{t'u_{ji}}{s^{k_{ji}}}\right),
$$

\noindent
$k_j,k_{ji}\in\N$, $u_{ji}\in R[t,t']$, $v_j\in IR[t,t']$ and $\alpha_j,\alpha_{ji}\in\Phi$.
Clearly, if $m\ge k_j,k_{ji}$ for all $j,i$, then
$$f\bigl(t,(s^mt')^9\bigr)\in\lambda_s\bigl(\E(\Phi,R[t,t'],It'R[t,t'])\bigr).$$

Take $n\ge 9\max(k_j,k_{ij})$ and put $t'=1$ to get
$\lambda_s\bigl(g(ts^n)\bigr)=f(t,s^n)=\lambda_s\bigl(h(t)\bigr)$ for some
$h(t)\in\E(\Phi,R[t],I[t])$.
By Lemma~\ref{ZeroDivisors} there exists $k\in\N$ such that
$g\bigl(ts^ks^n\bigr)=h(ts^k)\in\E(\Phi,R[t],IR[t])$, which completes the proof.
\end{proof}
\section{Relative Local-Global principle}

The proof of the LGP follows the same way as in Suslin's original proof.
However, the exposition is different. Namely, we do not use induction but introduce a
``generic ring'' for the problem. For us these arguments seem to be a little bit easier,
but this is a matter of taste.

Recall that $\lambda_\m$ denotes the localization homomorphism at a maximal ideal $\m$.

\begin{thm}\label{LG}
Let $I$ be an ideal of a commutative ring $R$ and $g=g(X)\in\G(\Phi,R[X],XR[X])$. Suppose that
$\lambda_\m(g)\in\E(\Phi,R_\m[X],IR_\m[X])$ for every maximal ideal $\m$ of $R$.

Then, $g\in\E(\Phi,R[X],IR[X])$.
\end{thm}

\begin{proof}
Denote by $S$ the set of all $s\in R$ such that $\lambda_{s}(g)\in\E(\Phi,R_s[X],IR_s[X])$.
We show that $S$ is unimodular.
Indeed, given a maximal ideal $\m$ of $R$ we can write $\lambda_\m(g)$ as a product
$$
\lambda_\m(g)=\prod_{j}x_{\alpha_j}\left( \frac{v_j}{r_j}\right)^{a_j}, \text{ where }
a_j=\prod_i x_{\alpha_{ji}}\left( \frac{u_{ji}}{r_{ji}}\right),
$$

\noindent
$r_j,r_{ji}\in R\setminus\m$, $u_{ji}\in R[X]$, $v_j\in IR[X]$ and $\alpha_j,\alpha_{ji}\in\Phi$.

Multiplying $r_j$ and $r_{ij}$ over all $i,j$ we get an element $s\in R\setminus\m$.
It is easy to see that $s\in S$.
Therefore, $S$ is not contained in any maximal ideal of $R$, hence it is unimodular.
Now, $1$ is a linear combination of a finite subset of elements $s_1,\dots,s_m\in S$ with
coefficients from $R$. Thus, we obtain a unimodular sequence $s_1,\dots,s_m$
such that
$$
\lambda_{s_k}(g)\in\E(\Phi,A_{s_k}[X],IA_{s_k}[X]) \text{ for all }k=1,\dots,m.
$$

Consider the ring $A=R[X,t_1,\dots,t_m]/(t_1+\dots+t_m=1)$.
Since $R[X]$ is a subring of $A$, we can view $g$ as an element of $\G(\Phi,A,XA)$.
Put
$$
g_1=g\Bigl(Xt_1\Bigr) \text{ and } g_k=g\Bigl(X\sum_{i=1}^{k-1}t_i\Bigr)^{-1}g\Bigl(X\sum_{i=1}^kt_i\Bigr)
\text{ for } k=2,\dots,m.
$$
Then $g_k\in\G(\Phi,A,t_kA)$ for all $k$ and $g=g_1\dots g_m$. Fix $k\in\{1,\dots,m\}$.
Note that $A=R^{(k)}[t_k]$ is a polynomial ring, where
$R^{(k)}=R[X,t_i\mid i\ne k,j]$ and $j$ is some index distinct from $k$
(in this case we view $t_j$ as an abbreviation of $1-\sum_{i\ne j}t_i$).
Since for any $s\in R$ localization homomorphism $\lambda_s$ commutes with evaluation homomorphisms
$X\mapsto X\sum_{i=1}^kt_i$, we have $\lambda_{s_k}(g_k)\in\E(\Phi,R^{(k)}_{s_k}[t_k],IR^{(k)}_{s_k}[t_k])$.
By Theorem~\ref{RelDilation} with $R=R^{(k)}$ and $t=t_k$, there exists $l_k\in\N$ such that
$g_k(s_k^{l_k}t_k)\in\E(\Phi,R^{(k)}[t_k],IR^{(k)}[t_k])$.

The sequence $s_1^{l_1},\dots,s_m^{l_m}$ is a unimodular sequence of elements of $R$
(otherwise all $s_i^{l_i}$ would lie in the same maximal ideal $\m$ which would imply that
all $s_i$'s are in $\m$, a contradiction).
Fix $p_1,\dots,p_m\in R$ such that $1=\sum_{k=1}^m s_k^{l_k}p_k$.
Denote by $\eps_k:A\to R^{(k)}$ the evaluation homomorphism sending $t_k$ to $s_k^{l_k}p_k$.
Then $\eps_k(g_k)\in\E(\Phi,R^{(k)},IR^{(k)})$.

Now, let $\eps:A\to R[X]$ be the homomorphism identical on $R[X]$ sending $t_k$ to $s_k^{l_k}p_k$
for all $k=1,\dots,m$
(by the choice of $p_1,\dots,p_m$ such a homomorphism exists).
Since $\eps$ factors through $\eps_k$ for each $k$, we have $\eps(g_k)\in\E(\Phi,A[X],IA[X])$,
hence $g=\eps(g)\in\E(\Phi,A[X],I[X])$, which completes the proof.
\end{proof}

\bibliographystyle{amsplain}
\bibliography{stepanov}
\end{document}